\documentclass[reqno]{amsart}
\usepackage{newcent}       
\usepackage{helvet}         
\usepackage{courier}        
\usepackage{amsfonts}
\usepackage{amsfonts,amssymb,amsmath}
\usepackage[latin1]{inputenc}
\usepackage{color}
\usepackage{graphicx}
\usepackage[active]{srcltx}
\newtheorem{theorem}{Theorem}[section]

\newtheorem{definition}{Definition}
\newcommand{\mc}[1]{{\mathcal #1}}

\newcommand{\bb}[1]{{\mathbb #1}}

\newcommand{\pfrac}[2]{\genfrac{}{}{}{1}{#1}{#2}}

\newcommand{\ket}{\rangle}
\newcommand{\ldois}{\ell^2(\bb Z)}
\newcommand{\mais}{|\!\!+\!\!1\rangle}
\newcommand{\menos}{|\!\!-\!\!1\rangle}
\newcommand{\tn}{\lfloor tn\rfloor}

\title[Hydrodynamic limit of quantum random walks]{Hydrodynamic limit of quantum random walks}

\author{Alexandre Baraviera }
\address{UFRGS, Instituto de Matem\'atica, Campus do Vale, Av. Bento Gon\c calves, 9500. CEP 91509-900, Porto Alegre, Brasil}
\curraddr{}
\email{baravi@mat.ufrgs.br}
\thanks{Partially supported by CNPq, CAPES and FAPERGS}

\author{Tertuliano Franco}
\address{UFBA\\
 Instituto de Matem\'atica, Campus de Ondina, Av. Adhemar de Barros, S/N. CEP 40170-110\\
Salvador, Brasil}
\email{tertu@impa.br}
\thanks{Partially supported by PRODOC-UFBA}

\author{Adriana Neumann}
\address{UFRGS, Instituto de Matem\'atica, Campus do Vale, Av. Bento Gon\c calves, 9500. CEP 91509-900, Porto Alegre, Brasil}
\curraddr{}
\email{aneumann@impa.br}
\thanks{}

\begin{document}

\begin{abstract}
We present a proof of the hydrodynamic limit of independent quantum random walks evolving on $\mathbb{Z}$. The time evolution of the local equilibrium is driven by a convolution of the initial profile with a rescaled version of the limiting  probability density obtained in \cite{gjs} in the law of large numbers for a single quantum random walk.
\end{abstract}

\maketitle

\section{Introduction}\label{s1}

An important subject in Statistical Physics is the comprehension of the hydrodynamic behavior of interacting particle systems. Roughly speaking, the hydrodynamic of a discrete system that evolves in time consists of proving the limit of the time trajectory of its spatial density of particles as some parameters are rescaled. These parameters can be, for instance, time and space. To prove rigorously such scaling limit is often a mathematical problem of deep technical difficulties. As a classical book in the subject we cite \cite{kl} and references therein.

The hydrodynamic limit has been developed and successfully proved for many interacting particle systems since the seventies, as the symmetric (and the asymmetric) exclusion process, the zero range process, independent random walks, and many others.

In the paper \cite{adz}, it was proposed a model for a quantum random walk. That model gave
origin to a vast literature,
inspiring several connections with quantum optics and quantum computation. We cite the survey \cite{kempe} on the subject of quantum random walks.

In \cite{gjs}, the authors proved a law of large numbers for a single quantum random walk. Differently from the classical random walk, the limit for the quantum random walk, in the ballistic scale, is not a deterministic number, but a probability distribution. This is in some sense a consequence of the fact the quantum random walk evolves faster than its classical version, say, after the same number of steps the average distance from the starting
point of a quantum walk is larger than its classical counterpart.

In this paper, we obtain the hydrodynamic limit for a system of independent quantum random walks;
to the best of our knowledge, it is
the first time a hydrodynamic limit is rigorously  obtained for a quantum system.
It is supposed that each quantum random walk starts from a localized state, and the numbers of quantum random walks starting at each state is determined by independent Poisson random variables. The parameter of each Poisson is a function of space, what is called \emph{slowly varying parameter}, driven by a smooth profile $\gamma$ of compact support.

Under these assumptions, we prove that the limiting profile is driven by a convolution of the initial profile $\gamma$ with the probability density obtained in \cite{gjs} in the law of large numbers for a single quantum random walk.

It is of worth to mention that, if the initial profile has compact support, then the limiting profile at any positive will also have compact support. This contrasts with the hydrodynamic limit for classical random walks, where the limiting profile evolves according to the heat equation. For the heat equation, it is well known that the diffusion has infinite propagation speed. That is, even for initial profiles with compact support, for any positive time, the solution will be different of zero for any point in the space.

In short,  we deduce in this work that, while  quantum random walks are faster than classical random walks in the scale aspect, they are slower in the macroscopic diffusion aspect.

The outline of the paper is the following:
in Section \ref{s2}, the model and statements of results are presented. In Section \ref{s3}, the hydrodynamic limit is proved.

\section{Model and statements}\label{s2}
The quantum random walk, to be defined ahead, is seen at discrete times $n\in \bb N$. In agreement with the postulates of the
Quantum Mechanics, its state is an element of a Hilbert space. Consider for now the Hilbert space of square summable complex double-sided sequences, or else, the space
\begin{equation*}
\ldois \;:=\;\{(\ldots,x_{-2},x_{-1},x_0,x_1,x_2\ldots)\,;\,\forall k\in \bb Z\,,\, x_k\in\bb C\,,\,\textrm{and}\,  \sum_{k\in \bb Z}|x_k|^2<\infty\}\,.
\end{equation*}
Let $\{e_k\}_{z\in \bb Z}$ be the canonical basis of $\ldois$. Thus, if $x\in\ldois$, then
$x=\sum_{k\in \bb Z} x_ke_k$.


If the state of the random walk at some time $n\in \bb N$ is some $x\in \ldois$, this means that, if we observe the quantum  random walk, then we will see him at some position $k\in \bb Z$ with probability proportional to $|x_k|^2$. In other words, the state of the quantum random walk is not a position, neither a probability vector. It is indeed a vector belonging to some Hilbert space, which in turn gives the information of finding the particle
at some point after a position measurement is performed in the system.

The state space of the quantum random walk here considered will not be  $\ldois$, but the space $\ldois\otimes\bb C^2$ instead, the tensor product of the Hilbert spaces $\ldois$ and $\bb C^2$. The space $\bb C^2$  represents the state of a quantum coin that will
be used to drive the quantum random walk as explained below.


We make a break for a few words on tensor products. If $f_1:H_1\to \bb C$ and $f_2:H_2\to \bb C$ are two linear functionals over some vector spaces $H_1$ and $H_2$, the tensor product of $f_1$ and $f_2$ is the bilinear functional $f_1\otimes f_2: H_1\times H_2\to \bb C$  defined by
\begin{equation*}
(f_1\otimes f_2)(y_1,y_2)\;:=\;f_1(y_1)\cdot f_2(y_2)\,.
\end{equation*}
The main difference in considering the tensor product of two Hilbert spaces\footnote{Recall that the Riesz Representation Theorem assures we can always see a Hilbert space as a space of linear functionals over itself.} instead of considering the cartesian product of them is in the fact that the tensor product is bilinear, whilst the cartesian product is linear.

According to the classical notation in Quantum Mechanics, the canonical basis of $\bb C^2$ will be denoted by $\{ \menos, \mais  \}$. An element of this Hilbert space $\bb C^2$ is called a \emph{qubit};
as said before, the qubit here is interpreted  as the state of a coin.
Elements of $\ldois$ from now on will be denoted by $|x\ket$. If the state of the random walk is at, let us say,
\begin{equation*}
\frac{1}{\sqrt{2}}\;|x\ket\otimes \menos +\frac{i}{\sqrt{2}}\; |y\ket\otimes \mais\,,
\end{equation*}
\noindent it can be understood as the system being
in the \emph{superposition of  states} $|x\ket \otimes \menos$ and $|y\ket \otimes \mais$,
a usual terminology in Physics.

Again, this quantum coin is not a probability vector. Supposing that the norm is unitary, in order to
obtain the probability of finding a particle at some specific point in $\bb Z$
we take the element of $\ldois$ representing the state,
and consider
the square of the respective component in the basis. In the example above, taking the square of each  component's norm, we conclude that the probability of finding the particle in $|x\ket$ or in $|y\ket$ is one half for each one.

In the literature is common to denote the over-all state of the quantum object
by $\mc H_P\otimes \mc H_C$, where $\mc H_P$ is the Hilbert space associated with the position and  $\mc H_C$ is the Hilbert space
associated with the state of a certain coin.
In our case, the most standard, $\mc H_P$ is $\ldois$ and $\mc H_C$ is $\bb C^2$.

The dynamics over this state space $\ldois\otimes\bb   C^2$ has two parts. Informally, the first part consists of a unitary operator\footnote{Unitary matrix: its columns (or lines) compound an orthonormal basis for the space.} that acts on the  coin. The second part is a translation to the right or to the left on the element of $\ldois$ depending if the respective coin qubit is $\mais$ or $\menos$.

We recall that the general form of an operator in $U_2(\bb C)$, the set of unitary operators, is
\begin{equation*}
     \left[
   \begin{array}{cc}
     a    &    \bar{b}  \\
     b    &    -\bar{a}
   \end{array}
  \right]\,,
\end{equation*}
\noindent where $a$ and $b$ are complex numbers such that $|a|^2 + |b|^2=1$. In this work, we deal with a particular operator $H\in U_2(\bb C)$,  the \emph{Hadamard} operator, given by
\begin{equation}\label{eq001}
   H\;:=\;
  \frac{1}{\sqrt{2}}\left[
   \begin{array}{cc}
     1    &    1 \\
     1    &    -1
   \end{array}
  \right]\,,
\end{equation}
whose effect emulates the evolution of an unbiased coin: for example, if the initial
coin state is $\mais$ then, after the action of $H$, we get $\frac{1}{\sqrt{2}}(\mais+\menos)$, meaning
that we have one half probability of finding one of the two possible coin states after
a measurement.

Now, we define the part of the dynamics that acts on space. Let $\tau_m:\ldois\to\ldois$ be the shift to the right of $m\in\bb Z$, or else, if $x=\sum_{k\in \bb Z} x_ke_k$, then
\begin{equation*}
\tau_k x\;:=\; \sum_{k\in \bb Z} x_ke_{k+m}\,.
\end{equation*}
The linear operator $S:\ldois\otimes \bb C^2 \to \ldois\otimes \bb C^2$ is defined through
\begin{equation*}
S \Big(|x\ket \otimes \mais \Big)\;:=\; |\tau_1 x\ket \otimes \mais\;,\quad \forall x\in\ldois\,,
\end{equation*}
and
\begin{equation*}
S \Big(|x\ket \otimes \menos \Big)\;:=\; |\tau_{-1} x\ket \otimes \menos\;,\quad \forall x\in\ldois\,.
\end{equation*}
 Finally, denote $Id$ the identity operator over $\ldois$ and define $U:\ldois\otimes \bb C^2 \to \ldois\otimes \bb C^2$ by the composition
\begin{equation*}
 U \;: = \; S \circ (Id \otimes H)\,.
\end{equation*}
The dynamics is the following: if at time zero  the state is some $\psi\in\ldois\otimes \bb C^2$,
then the state at time $n=1$ is given by $U(\psi)$, and at arbitrary time $n\in \bb N$ is given by  $U^{(n)}(\psi)$.

To exemplify the model, in our case, where it is fixed the Hadamard operator $H$, for
\begin{equation*}
 \psi \;= \;|x\ket \otimes \mais + |y\ket \otimes \menos\,,
\end{equation*}
we have that
\begin{equation}\label{eq:01}
\begin{split}
U(\psi) & \;= \;S\Big(|x\ket \otimes H(\mais) + |y\ket \otimes H(\menos)\Big)\\
	  & \;= \;S\Big(|x\ket \otimes \Big(\pfrac{|+1\ket +|-1\ket}{\sqrt{2}}\Big) + |y\ket \otimes  \Big(\pfrac{|+1\ket -|-1\ket}{\sqrt{2}}\Big)\Big)\\
	  & \;= \;\frac{1}{\sqrt{2}}\Big[\Big(|\tau_1 x\ket+|\tau_1 y\ket\Big)  \otimes \mais + \Big(|\tau_{-1} x\ket-|\tau_{-1} y\ket\Big)  \otimes \menos\Big)\Big]\,.\\
\end{split}
\end{equation}
Notice that the dynamics is indeed deterministic. On the other hand, as aforementioned, the position of the particle after an observation is a random variable, being its distribution obtained from the state at the time of the observation. Therefore,  given an initial state $\psi\in  \ldois\otimes \bb C^2$, the state $U^{(n)}(\psi)$ obtained after $n$ iterations gives all the information about the distribution of where the particle is at time $n$ (if an observation is performed).

It is a common sense in Quantum Mechanics that the particle is not at any particular place before the observation. Only after the observation, and thus after the consequently random result, one can say that the particle is at some place
(more sophisticate interpretations are available but here we adopt this very pragmatic point of view).

Two remarks: although the Hilbert spaces are complex, if we multiply the initial state by a complex number, there is no changing in the particle space distribution at final time. Or else, for non-zero $\zeta\in \bb C$, the distributions of position at  time $n\in \bb N$, obtained from $U^{(n)}(\psi)$ and from $U^{(n)}(\zeta \psi)$ are the same. The role of complex numbers is noted in Quantum Mechanics  when  considering sums of states (leading to  phenomena known as \emph{interference}).
Secondly, the signs appearing in \eqref{eq:01} originate cancellations and hence a very peculiar behavior\footnote{In comparison with the classical random walk.} of the quantum random walk, as commented in Section \ref{s1}.

Next, we explain what we mean by a system of independent quantum random walks. For $\psi\in\ldois\otimes \bb C^2$, we have that
\begin{equation}\label{eq:02}
U^{(n)}(\psi)\;=\;\Big(\sum_{k\in\bb Z}x_k^n e_k\Big)\otimes \mais + \Big(\sum_{k\in\bb Z}y_k^n e_k\Big)\otimes \menos\,,
\end{equation}
where $x_k^n,y_k^n\in \bb C$ depend on $\psi$.
\begin{definition}\label{def:01}
For $\psi\in\ldois\otimes \bb C^2$ with unitary norm, let $X_n^\psi$ be a random variable assuming integers values  with distribution given by
\begin{equation*}
 \bb P\Big( X_n^\psi=k\Big)\;=\;|x_k^n|^2+|y_k^n|^2\,,\qquad\forall\, k\in \bb Z\,,
\end{equation*}
being $x_k^n,y_k^n$  those ones given in \eqref{eq:02}.
\end{definition}

Let $\Psi:{\ldois\otimes \bb C^2}\to\bb N$ be a configuration of quantum random walks, where
$\Psi(\psi)$ denotes how many quantum random walks are in the state $\psi\in\ldois\otimes\bb C^2$ for the configuration $\Psi$.
\begin{definition}\label{def:02}
Let $\Psi$ be a configuration of quantum random walks. For each state $\psi\in\ldois\otimes\bb C^2$, let
$X_n^\psi(1),\ldots,X^\psi_n(\Psi(\psi))$ be quantum random walks starting from the state $\psi$, all of them independent.
Define the function $\eta_n^\Psi:\bb Z\to\bb N$ by
\begin{equation}\label{eq03.1}
\eta_n^\Psi(k)\;=\;\sum_{\psi}\sum_{j=1}^{\Psi(\psi)}{\bf{1}}{[X_n^\psi(j)=k]}\,.
\end{equation}
where the sum in $\psi$ is taken over $\ldois\otimes\bb C^2$.
\end{definition}
The function $\eta_n^\Psi$ counts the quantity  of particles  at each site of $\bb Z$ after observations at time $n$ of each  quantum random walk, assuming that the system started from the configuration $\Psi$ of quantum random walks.

Of course, there is no reason to be finite the sum in \eqref{eq03.1}. For this reason, we are going to choose $\Psi$ having a finite number of quantum random walks, or else, $\Psi(\psi)=0$ except for a finite number of states $\psi$. This will allow future interchanging of sums and integrals.

Let $\gamma:\bb R\to \bb R_+$ be a fixed continuous profile with compact support.
\begin{definition}\label{def:03} Let
  $\Theta_\gamma^n$ be the product measure on $\bb N^{\ldois\otimes \bb C^2}$ which marginal at the state $\psi$ is a measure of type
\begin{equation}\label{eq04}
\begin{cases}
\textrm{Poisson}\Big(\frac{1}{2}\gamma(\frac{k}{n})\Big)\,,&\quad \textrm{ if }\, \psi= e_k\otimes \mais\,,\\
\textrm{Poisson}\Big(\frac{1}{2}\gamma(\frac{k}{n})\Big)\,,&\quad \textrm{ if }\, \psi= e_k\otimes \menos\,,\\
\delta_0\,,&\quad \textrm{ otherwise}\,.
\end{cases}
\end{equation}
\end{definition}
In words, if the state is of the form $\psi= e_k\otimes |\!+\!1\ket$ or $\psi= e_k\otimes |\!-\!1\ket$, then it is chosen a quantity of particles with this initial state $\psi$ through a Poisson measure, which parameter depends on $\gamma$. This is usually called a \emph{slowly varying parameter} in hydrodynamic limit, see \cite{kl}. No other initial states are allowed\footnote{This is not indispensable, but we chose this restriction for sake of clarity.}. Denote by $\bb E_{\Theta_\gamma^n}$ and $\bb P_{\Theta_\gamma^n}$ the probability and expectation, respectively,  induced by choosing of the initial state according to $\Theta_\gamma^n$. \medskip

Denote be $C_c(\bb R)$ the set of continuous functions $H:\bb R\to\bb R$ with compact support.
\begin{definition}\label{def:05}
Let  $f:\bb R\to\bb R_+$ be the probability density function given by
\begin{equation*}
 f(x)\;=\;\begin{cases}
           \frac{1}{\pi(1-x^2)\sqrt{1-2x^2}}\,,&\textrm{ if } x\in [-\frac{\sqrt{2}}{2},\frac{\sqrt{2}}{2}]\,,\\
           0\,,& \textrm{ otherwise}\,.
          \end{cases}
\end{equation*}
Denote also $f_t(x):=\pfrac{1}{t}f(\pfrac{x}{t})$.
\end{definition}
\noindent
We are in position to the state our main result. For $y\in \bb R$, denote by $\lfloor y\rfloor $ the greatest integer less or equal than $y$.
\begin{theorem}[Hydrodynamic limit of quantum random walks]\label{th:2.1}
Let $\eta_n^\Psi$ be the random element given in Definition \ref{def:02} with initial configuration $\Psi$  chosen according to the  measure $\Theta_\gamma^n$ given in Definition \ref{def:03}.

Then, for all $t> 0$, and for all $H\in C_c(\bb R)$,
\begin{equation*}
 \lim_{n\to\infty} \frac{1}{n} \sum_{k\in\bb Z}H(\pfrac{k}{n})\,\eta_{\lfloor t n\rfloor}^\Psi(k)\;=\; \int_{\bb R} H(x)\rho(t,x)dx
\end{equation*}
in probability, where the function $\rho: \bb R_+\times \bb R\to \bb R_+$ is given by
\begin{equation}\label{eq:03}
 \rho(t,x)\;=\;(\gamma *f_t)(x)\;:=\;\int_{\bb R} \gamma(y)\,\pfrac{1}{t}f(\pfrac{x-y}{t})\,dy \,,
\end{equation}
being $f$ given in the Definition \ref{def:05}.
\end{theorem}
 We notice that the time $\lfloor t n\rfloor$ appearing  in the statement above corresponds to the so-called \emph{ballistic} scaling. In models where it appears $\lfloor t n^2\rfloor$ instead, it is called \emph{diffusive} scaling. Despite the Hadamard operator being an honest coin in some sense, the limit for a system of independent quantum random walks is driven in the ballistic scale, while for classical random walks it would be driven in the diffusive scale. This is an intrinsic characteristic of quantum random walks: because of cancellations aforementioned, they move faster than classical random walks.

 On the other hand, the time evolution of initial profile $\gamma$ according to $\gamma*f_t$ is somewhat slower than the equivalently time evolution obtained in the case of classical independent random walks, where the initial profile $\gamma$ would evolve through the heat equation's semigroup. For any positive time, the solution of heat equation starting from $\gamma$ is positive everywhere, which does
not happen with $\gamma*f_t$. Since $f$ and $\gamma$ have compact support, for any time $t>0$, the function $\gamma*f_t$ has compact support as well. This means that the diffusion of mass through the $\gamma*f_t$ has finite velocity of propagation, differently of the diffusion by means of the heat equation.

\section{Hydrodynamic limit}\label{s3}

We prove indeed something slightly stronger than the hydrodynamic limit stated in the Theorem \ref{th:2.1}. In this section, we obtain a result usually called as \emph{conservation of local equilibrium}, which is his hand implies the hydrodynamic limit as stated in the Theorem \ref{th:2.1}.

\subsection{Local equilibrium}
We begin with some topological considerations.
In the space $\bb N^{\bb Z}$ endowed with the distance
\begin{equation*}
d(\eta_1,\eta_2)\;=\; \sum_{k\in \bb Z}\frac{1}{2^{|k|}}\frac{|\eta_1(k)-\eta_2(k)|}{1+|\eta_1(k)-\eta_2(k)|}\,,
\end{equation*}
denote by $\{\tau_k\;;\; k\in \bb Z\}$ the group of translations. In others words, $\tau_k\eta$ is the configuration such that
\begin{equation*}
(\tau_k\eta)(j)\;=\;\eta(j+k)\,.
\end{equation*}
The action of the translation group is naturally extended to the space of probability measures on $\bb N^\bb Z$. For $k\in \bb Z$ and a probability measure $\mu$ on $\bb N^{\bb Z}$, we denote by $\tau_k\mu$ the unique probability measure such that
\begin{equation*}
 \int f(\eta)(\tau_k\mu)(d\eta)\;=\;\int f(\tau_k\eta)\mu(d\eta)
\end{equation*}
for all integrable continuous functions $f$ in the topology induced by the aforementioned distance.

For $\alpha>0$, define $\nu_\alpha$ as the probability product measure on $\bb N^{\bb Z}$ which marginals are Poisson probabilities measures of same parameter $\alpha>0$, i.e,
\begin{equation*}
\nu_{\alpha}\{\eta\;;\; \eta(k)=\ell\}\;=\; e^{-\alpha} \frac{\alpha^\ell}{\ell!}\,,
\end{equation*}
for any $k\in \bb Z$.

Informally, the conservation of local equilibrium says that under suitable hypothesis about the initial distribution of particles, the distribution of observed particles at time $\lfloor tn\rfloor$ is, in a asymptotically sense, locally a Poisson product measure which parameter is a function of the time and space.
Its statement is
\begin{theorem}[Conservation of local equilibrium]\label{th:2.2}
 Let $\mu^{\Psi}_{\lfloor tn\rfloor}$ be the probability measure on $\bb N^{\bb Z}$ induced by the random element $\eta^\Psi_{\lfloor tn\rfloor}$ given in the Definition \eqref{eq03.1}. Then, for any $x\in \bb R$ and any $t>0$,
 \begin{equation*}
 \lim_{n\to\infty} \tau_{\lfloor xn\rfloor} \mu^{\Psi}_{\lfloor tn\rfloor}\;=\;\nu_{\rho(t,x)}
 \end{equation*}
in the sense of weak convergence of probabilities measures, see the reference \cite{bill}, and $\rho(t,x)$ is defined in
\eqref{eq:03}.
 \end{theorem}
\begin{proof}
The weak convergence of probability measures on $\bb N^{\bb Z}$ is equivalent to the convergence of its finite dimensional distributions, see \cite{bill}. On his hand, the convergence finite dimensional distribution is characterized by the convergence of its Laplace transform. In this way, let $\lambda:\bb Z\to \bb R_+$ be a function such that $\lambda(k)\neq 0$ only for a finite subset of $\bb Z$.  We are therefore concerned with the Laplace transform
\begin{equation*}
\bb E_{\Theta_\gamma^n}\Big[\exp\Big(-\sum_{k\in \bb Z}\lambda(k)\eta^{\Psi}_{\lfloor tn\rfloor}(k)\Big)\Big]\,.
\end{equation*}
By \eqref{eq03.1}, we get that
\begin{equation*}
\begin{split}
\sum_{k\in \bb Z}\lambda(k)\eta^{\Psi}_{\lfloor tn\rfloor}(k) & \;=\;\sum_{k\in \bb Z}\lambda(k)\sum_{\psi}\sum_{j=1}^{\Psi(\psi)}{\bf{1}}{[X_{\tn}^\psi(j)=k]}\;=\;\sum_{\psi}\sum_{j=1}^{\Psi(\psi)}\lambda(X_{\tn}^\psi(j))\,,\\
\end{split}
\end{equation*}
being the sum in $\psi$ taken over $\ldois\otimes\bb C^2$.
Since the quantum random walks are taken as independent and by the choose of $\Psi$ ($\Psi(\psi)=0$ except for a finite number of states $\psi$), we have
\begin{equation}\label{eq:06}
\begin{split}
 & \bb E_{\Theta_\gamma^n}\Big[\exp\Big(-\sum_{k\in \bb Z}\lambda(k)\eta^{\Psi}_{\lfloor tn\rfloor}(k)\Big)\Big] \;=\;
 \prod_{\psi} \bb E_{\Theta_\gamma^n}\Big[\exp\Big(-\sum_{j=1}^{\Psi(\psi)}\lambda(X_{\tn}^\psi(j))\Big)\Big] \\
 & =\;
 \prod_{\psi} \int E \Big[\exp\Big(-\lambda(X_{\tn}^\psi(1))\Big)\Big]^{\Psi(\psi)} d\Theta_\gamma^n(\Psi)\,,
 \\
\end{split}
\end{equation}
where the expectation $E$ above is over a single quantum random walk. Denote
\begin{equation*}
\beta(\psi):= E \Big[\exp\Big(-\lambda(X_{\tn}^\psi(1))\Big)\Big]\,.
\end{equation*}
Under $\Theta_\gamma^n$, we have that $\Psi(\psi)$ has Poisson  distribution given by \eqref{eq04}. Hence,
\begin{equation*}
\int \beta(\psi)^{\Psi(\psi)} d\Theta_\gamma^n(\Psi)\;=\;
\begin{cases}
\exp\Big(\pfrac{1}{2}\gamma(\frac{k}{n})(\beta(\psi)-1)\Big)\,,&\quad \textrm{ if }\, \psi= e_k\otimes \mais\,,\\
\exp\Big(\pfrac{1}{2}\gamma(\frac{k}{n})(\beta(\psi)-1)\Big)\,,&\quad \textrm{ if }\, \psi= e_k\otimes \menos\,,\\
1\,,&\quad \textrm{ otherwise}\,.
\end{cases}
\end{equation*}
We notice that $\beta(\psi)=E\Big[\exp\Big(-\lambda(X_{\tn}^\psi(1))\Big)\Big]$ is the same for $\psi= e_k\otimes \mais$ or
$\psi= e_k\otimes \menos$. This is a consequence of the symmetry of the Hadamard operator defined in \eqref{eq001}. Since those states are the meaning ones, we can index only on them. Henceforth, denote
\begin{equation*}
\beta_k\;:= \;\beta(e_k\otimes \mais)\;=\; \beta(e_k\otimes \menos)\,.
\end{equation*}
Applying these considerations to  \eqref{eq:06}, it leads to
\begin{equation}\label{eq07}
\begin{split}
  \bb E_{\Theta_\gamma^n}\Big[\exp\Big(-\sum_{k\in \bb Z}\lambda(k)\eta^{\Psi}_{\lfloor tn\rfloor}(k)\Big)\Big] &\;=\;
 \prod_{k\in \bb Z} \exp\Big(\gamma(\pfrac{k}{n})(\beta_k-1)\Big) \\
 &\;=\;
  \exp\Big(\sum_{k\in \bb Z}\gamma(\pfrac{k}{n})(\beta_k-1)\Big)\,. \\
\end{split}
\end{equation}
Denote by $p_n(k,j)$ the probability of the quantum random walk, starting at $\psi=e_k\otimes \mais$ or
$\psi=e_k\otimes \menos$, be observed at the position $j\in \bb Z$ after a time $n$.
Thus, for $\psi=e_k\otimes \mais$ or $\psi=e_k\otimes \menos$,
\begin{equation*}
\beta_k\;=\;E \Big[\exp\Big(-\lambda(X_{\tn}^\psi(1))\Big)\Big]\;=\;\sum_{j\in \bb Z}e^{-\lambda(j)}p_{\tn}(k,j)\,.
\end{equation*}
Replacing the formula above for $\beta_k$ in \eqref{eq07} and changing the order of summation,
\begin{equation*}
\begin{split}
  \bb E_{\Theta_\gamma^n}\Big[\exp\Big(-\sum_{k\in \bb Z}\lambda(k)\eta^{\Psi}_{\lfloor tn\rfloor}(k)\Big)\Big] &\;=\;
\exp\Big(\sum_{j\in \bb Z}(e^{-\lambda(j)}-1)\sum_{k\in \bb Z}\gamma(\pfrac{k}{n})p_{\tn}(k,j)\Big)\,.\\
 \end{split}
\end{equation*}
The formula above characterizes the measure $\mu^{\Psi}_{\lfloor tn\rfloor}$ on $\bb N^{\bb Z}$ induced by $\eta^\Psi_{\lfloor tn\rfloor}$ as a product measure which marginal at $j\in\bb Z$ is a Poisson probability measure of parameter
\begin{equation*}
B(j,\tn)\;:=\;\sum_{k\in \bb Z}\gamma(\pfrac{k}{n})p_{\tn}(k,j)\;=\;\sum_{k\in \bb Z}\gamma(\pfrac{k}{n})p_{\tn}(j,k)\;=\;
E\Big[\gamma\Big(\frac{X^j_{\tn}}{n}\Big)\Big]\,,
\end{equation*}
being the interchanging of $j$ and $k$ above a consequence of  Hadamard operator's symmetry. Since $\gamma$ is smooth,
\begin{equation*}
\begin{split}
\lim_{n\to\infty}B(\lfloor xn\rfloor ,\lfloor t n\rfloor)\;=\;
\lim_{n\to\infty}E\Big[\gamma\Big(\frac{X^{\lfloor xn\rfloor}_{\lfloor t n\rfloor}}{n}\Big)\Big]&\;=\;
\lim_{n\to\infty}E\Big[\gamma\Big(\frac{t\,X^{\lfloor xn\rfloor}_{\lfloor t n\rfloor}}{tn}\Big)\Big]
\\
&\;=\; \int_{\bb R} \gamma(ty+x)\,f(y)\,dy \\
&\;=\; \int_{\bb R} \gamma(y)\,\pfrac{1}{t}f(\pfrac{x-y}{t})\,dy\,, \\
\end{split}
\end{equation*}
by the law of large numbers for quantum random walks, see \cite[Theorem 1]{gjs}, and the fact that $f$ is an even function. The limit above assures that
 \begin{equation*}
 \lim_{n\to\infty} \tau_{\lfloor xn\rfloor} \mu^{\Psi}_{\lfloor tn\rfloor}\;=\;\nu_{\rho(t,x)} \,,
 \end{equation*}
 finishing the proof.
\end{proof}

\begin{proof}[Proof of Theorem \ref{th:2.1}]
It is a known result that the conservation of local equilibrium, proved in Theorem \ref{th:2.1}, implies the hydrodynamic limit stated in the Theorem \ref{th:2.2}. See \cite[chapter 3]{kl}, for instance.
\end{proof}


\section*{Acknowledgements}
T.F would like to thank the Department of Mathematics of UFRGS, Porto Alegre - Brazil, where this work was initiated, by its hospitality.

\end{document}